\documentclass[12pt,reqno]{amsart}
\usepackage{amsmath}
\usepackage{amssymb}
\usepackage[left=3cm,top=3cm,right=3cm,bottom=3cm]{geometry}
\usepackage{graphicx}
\usepackage[usenames]{color}
\usepackage{enumerate}

\begin{document}
\newtheorem{theorem}{Theorem}[section]
\newtheorem{lemma}[theorem]{Lemma}
\newtheorem{definition}[theorem]{Definition}
\newtheorem{conjecture}[theorem]{Conjecture}
\newtheorem{proposition}[theorem]{Proposition}
\newtheorem{claim}[theorem]{Claim}
\newtheorem{corollary}[theorem]{Corollary}
\newtheorem{observation}[theorem]{Observation}
\newtheorem{problem}[theorem]{Open Problem}

\newtheorem*{theorem_a}{Theorem \ref{thm:main3}a}
\newtheorem*{theorem_b}{Theorem \ref{thm:main3}b}

\theoremstyle{remark}
\newtheorem{rem}[theorem]{Remark}

\newcommand{\noin}{\noindent}
\newcommand{\ind}{\indent}
\newcommand{\om}{\omega}
\newcommand{\pp}{\mathcal P}
\newcommand{\AC}{\mathcal A \mathcal C}
\newcommand{\bAC}{\overline{\AC}}
\newcommand{\ppp}{\mathfrak P}
\newcommand{\N}{{\mathbb N}}
\newcommand{\LL}{\mathbb{L}}
\newcommand{\R}{{\mathbb R}}
\newcommand{\E}{\mathbb E}
\newcommand{\Prob}{\mathbb{P}}
\newcommand{\eps}{\varepsilon}
\newcommand{\G}{{\mathcal{G}}}

\newcommand{\Ss}{{\mathcal S}}
\newcommand{\Nn}{{\mathcal N}}

\newcommand{\ceil}[1]{\left \lceil #1 \right \rceil}
\newcommand{\floor}[1]{\left \lfloor #1 \right \rfloor}
\newcommand{\size}[1]{\left \vert #1 \right \vert}
\newcommand{\dist}{\mathrm{dist}}

\title{Acquaintance time of random graphs near connectivity threshold}

\author{Andrzej Dudek}
\address{Department of Mathematics, Western Michigan University, Kalamazoo, MI, USA}
\email{\tt andrzej.dudek@wmich.edu}
\thanks{The first author is supported in part by Simons Foundation Grant \#244712 and by a grant from the Faculty Research and Creative Activities Award (FRACAA), Western Michigan University.}

\author{Pawe\l{} Pra\l{}at}
\address{Department of Mathematics, Ryerson University, Toronto, ON, Canada}
\email{\tt pralat@ryerson.ca}
\thanks{The second author is supported in part by NSERC and Ryerson University}

\thanks{Work partially done during a visit to the Institut Mittag-Leffler (Djursholm, Sweden)}


\begin{abstract}
Benjamini, Shinkar, and Tsur stated the following conjecture on the acquaintance time: asymptotically almost surely $\AC(G) \le p^{-1} \log^{O(1)} n$ for a random graph $G \in G(n,p)$, provided that $G$ is connected. Recently, Kinnersley, Mitsche, and the second author made a major step towards this conjecture by showing that asymptotically almost surely $\AC(G) = O(\log n / p)$, provided that $G$ has a Hamiltonian cycle. In this paper, we finish the task by showing that the conjecture holds in the strongest possible sense, that is, it holds right at the time the random graph process creates a connected graph. Moreover, we generalize and investigate the problem for random hypergraphs.
\end{abstract}

\maketitle

\section{Introduction}\label{sec:intro}

In this paper, we study the following graph process, which was recently introduced by Benjamini, Shinkar, and Tsur~\cite{bst}. Let $G=(V,E)$ be a finite connected graph. We start the process by placing one \emph{agent} on each vertex of $G$. Every pair of agents sharing an edge is declared to be \emph{acquainted}, and remains so throughout the process. In each round of the process, we choose some matching $M$ in $G$.  ($M$ need not be maximal; perhaps it is a single edge.) For each edge of $M$, we swap the agents occupying its endpoints, which may cause more agents to become acquainted. The \emph{acquaintance time} of $G$, denoted by $\AC(G)$, is the minimum number of rounds required for all agents to become acquainted with one another.

It is clear that 
\begin{equation}\label{eq:trivial_lower}
\AC(G) \ge \frac {{|V| \choose 2}}{|E|} - 1, 
\end{equation}
since $|E|$ pairs are acquainted initially, and at most $|E|$ new pairs become acquainted in each round. In~\cite{bst}, it was shown that always $\AC(G) = O(\frac{n^2}{\log n / \log \log n})$, where $n = |V|$. A small progress was made in~\cite{KMP} where it was proved that $\AC(G) = O(n^2/\log n )$. This general upper bound was recently improved in~\cite{AS} to $\AC(G) = O(n^{3/2})$, which was conjectured in~\cite{bst} and is tight up to a multiplicative constant. Indeed, for all functions $f : \N \rightarrow \N$ with $1 \le f(n) \le n^{3/2}$, there are families $\{G_n\}$ of graphs with $\size{V(G_n)} = n$ for all $n$ such that $\AC(G_n) = \Theta(f_n)$. The problem is similar in flavour to the problems of Routing Permutations on Graphs via Matchings~\cite{ACG94}, Gossiping and Broadcasting~\cite{HHL88}, and Target Set Selection~\cite{KKT03, Che09, Rei12}.

\bigskip

A \emph{hypergraph} $H$ is an ordered pair $H=(V,E)$, where $V$ is a finite set (the \emph{vertex set}) and $E$ is a family of disjoint subsets of $V$ (the \emph{edge set}). A hypergraph $H=(V,E)$ is \emph{$r$-uniform} if all edges of $H$ are of size $r$.
In this paper, we consider the acquaintance time of random $r$-uniform hypergraphs, including binomial random graphs. For a given $r \in \N \setminus \{1\}$, the \emph{random $r$-uniform hypergraph} $H_r(n,p)$ has $n$ labelled vertices from a vertex set $V = [n] :=\{1,2,\dots,n\}$, in which every subset $e \subseteq V$ of size $|e|=r$ is chosen to be an edge of $H$ randomly and independently with probability $p$. For $r=2$, this model reduces to the well known and thoroughly studied model $G(n,p)$ of \emph{binomial random graphs}. (See, for example, \cite{bol,JLR}, for more.)  Note that $p=p(n)$ may tend to zero (and usually does) as $n$ tends to infinity. All asymptotics throughout are as $n \rightarrow \infty $ (we emphasize that the notations $o(\cdot)$ and $O(\cdot)$ refer to functions of $n$, not necessarily positive, whose growth is bounded). We say that an event in a probability space holds \emph{asymptotically almost surely} (or \emph{a.a.s.}) if the probability that it holds tends to $1$ as $n$ goes to infinity. 

\smallskip

Let $G \in G(n,p)$ with $p = p(n) \ge (1-\eps) \log n / n$ for some $\eps > 0$. Recall that $\AC(G)$ is defined only for connected graphs, and $\log n / n$ is the sharp threshold for connectivity in $G(n,p)$---see, for example,~\cite{bol, JLR} for more. Hence, we will not be interested in sparser graphs. It follows immediately from Chernoff's bound that a.a.s.\ $|E(G)| = (1+o(1)) {n \choose 2} p$. Hence, from the trivial lower bound~(\ref{eq:trivial_lower}) we have that a.a.s.\ $\AC(G) = \Omega(1/p)$. Despite the fact that no non-trivial upper bound on $\AC(G)$ was known, it was conjectured in~\cite{bst} that a.a.s.\ $\AC(G) = O(\text{poly} \log(n) / p)$. Recently, Kinnersley, Mitsche, and the second author of this paper made a major step towards this conjecture by showing that a.a.s.\ $\AC(G) = O(\log n / p)$, provided that $G$ has a Hamiltonian cycle, that is, when $pn - \log n -\log \log n \to \infty$. (See~\cite{KMP} for more.)

In this paper, we finish the task by showing that the conjecture holds above the threshold for connectivity.

\begin{theorem}\label{thm:main2}
Suppose that $p=p(n)$ is such that $pn - \log n \to \infty$. For $G \in G(n,p)$, a.a.s.
$$
\AC(G) = O \left( \frac {\log n}{p} \right).
$$
\end{theorem}

In fact, we prove the conjecture in the strongest possible sense, that is, we show that it holds right at the time the random graph process creates a connected graph. Before we state the result, let us introduce one more definition. We consider the Erd\H{o}s-R\'enyi \emph{random graph process}, which is a stochastic process that starts with $n$ vertices and no edges, and at each step adds one new edge chosen uniformly at random from the set of missing edges. Formally, let $N={n \choose 2}$ and let $e_1, e_2, \ldots, e_N$ be a random permutation of the edges of the complete graph $K_n$. The graph process consists of the sequence of random graphs $( \G(n,m) )_{m=0}^{N}$, where $\G(n,m)=(V,E_m)$, $V = [n]$, and $E_m = \{ e_1, e_2, \ldots, e_m \}$. It is clear that $\G(n,m)$ is a graph taken uniformly at random from the set of all graphs on $n$ vertices and $m$ edges. (As before, see, for example,~\cite{bol, JLR} for more details.)

Note that in the missing window in which the conjecture is not proved, we have $p=(1+o(1)) \log n / n$ and so an upper bound of $O(\log n / p)$ is equivalent to $O(n)$. Hence, it is enough to show the following result, which implies Theorem~\ref{thm:main2} as discussed at the beginning of Section~\ref{sec:graphs}.

\begin{theorem}\label{thm:main1}
Let $M$ be a random variable defined as follows:
$$
M = \min \{ m : \text{$\G(n,m)$ is connected} \}.
$$
Then, for $G \in \G(n,M)$, a.a.s.
$$
\AC(G) = O \left( n \right).
$$
\end{theorem}

\bigskip

Fix $r \in \N \setminus \{1\}$. For an $r$-uniform hypergraph $H=(V,E)$, there are a number of ways one can generalize the problem. At each step of the process, we continue choosing some matching $M$ of agents and swapping the agents occupying its endpoints; two agents can be matched if there exists $e \in E$ such that both agents occupy vertices of $e$. However, this time we have more flexibility in the definition of being acquainted. We can fix $k \in \N$ such that $2 \le k \le r$, and every $k$-tuple of agents sharing an edge is declared to be \emph{$k$-acquainted}, and remains so throughout the process. The \emph{$k$-acquaintance time} of $H$, denoted by $\AC^k_r(H)$, is the minimum number of rounds required for all $k$-tuples of agents to become acquainted. Clearly, for any graph $G$, $\AC^2_2(G) = \AC(G)$.

\smallskip

We show the following result for random $r$-uniform hypergraphs. Note that $p=\frac {(r-1)!\log n}{n^{r-1}}$ is the sharp threshold for connectivity so the assumption for the function $p$ in the statement of the result is rather mild but could be weakened. 

\begin{theorem}\label{thm:main3}
Let $r \in \N \setminus \{1\}$, let $k \in \N$ be such that $2 \le k \le r$, and let $\eps > 0$ be any real number. Suppose that $p=p(n)$ is such that $p \ge (1+\eps) \frac {(r-1)!\log n}{n^{r-1}}$. For $H \in H_r(n,p)$, a.a.s.
$$
\AC^k_r(H) = \Omega \left( \max \left\{ \frac {1}{p n^{r-k}},1 \right\} \right)
\quad \text{ and } \quad
\AC^k_r(H) = O \left( \max \left\{  \frac {\log n}{p n^{r-k}},1 \right\} \right).
$$
\end{theorem}

\bigskip

Throughout the paper, we will be using the following concentration inequalities. Let $X \in \textrm{Bin}(n,p)$ be a random variable with the binomial distribution with parameters $n$ and $p$. Then, a consequence of Chernoff's bound (see e.g.~\cite[Corollary~2.3]{JLR}) is that 
\begin{equation}\label{chern}
\Prob( |X-\E X| \ge \eps \E X) \le 2\exp \left( - \frac {\eps^2 \E X}{3} \right)  
\end{equation}
for  $0 < \eps < 3/2$. This inequality will be used many times but at some point we will also apply  the bound of Bernstein (see e.g.~\cite[Theorem~2.1]{JLR}) that for every $x > 0$, 
\begin{equation}\label{Bernstein}
\Prob \left( X \ge (1+x) \E X \right) \le \exp \left( - \frac {x^2  \E X} {2(1+x/3)} \right).
\end{equation}

\bigskip

The paper is structured as follows. In Section~\ref{sec:paths}, we investigate the existence of long paths in $H_r(n,p)$. Our result works for any $r \ge 2$ and will be used in both graphs and hypergraphs. In Section~\ref{sec:graphs}, we deal with random graphs ($r=2$) and random hypergraphs ($r \ge 3$) are dealt in Section~\ref{sec:hypergraphs}. We finish the paper with a few remarks.

\section{Existence of long paths in $H_r(n,p)$}\label{sec:paths}

For positive integers $k,r$ and $\ell$ satisfying $\ell = k(r-1) + 1$, we say that an $r$-uniform hypergraph $L$ is a \emph{loose path} of length $\ell$ if there exist an ordering of the vertex set, $(v_1, v_2, \ldots, v_{\ell})$, and an ordering of the edge set, $(e_1, e_2, \ldots, e_{k})$, such that each edge consists of $r$ consecutive vertices, that is, for every $i$, $e_i = \{ v_{(i-1)(r-1)+q} : 1 \le q \le r \}$, and each pair of consecutive edges intersects on a single vertex, that is, for every $i$, 
$e_i \cap e_{i+1} = \{ v_{i(r-1)+1} \}$. Note that the ordering $(v_1, v_2, \ldots, v_{\ell})$ of vertices can be used to completely describe $L$. 

\begin{lemma}\label{lem:alg}
For fixed $r \in \N \setminus \{1\}$ and real number $0<\delta<1$, there is a positive constant $c=c(r,\delta)$ such that for
$H\in H_r(n,p)$ with $p=c/n^{r-1}$
\[
\Prob\left(H \text{ has a loose path of length at least $\delta n$}\right) \ge 1-\exp(-n).
\]
\end{lemma}

\noindent
(For expressions such as $\delta n$ that clearly have to be an integer, we round up or down but do not specify which: the reader can choose either one, without affecting the argument.)

\begin{proof}
We will apply a depth first search algorithm in order to explore vertices of $H=(V,E) \in H_r(n,p)$ extending some ideas from~\cite{KLS}.
Set
\[
c= c(r, \delta) = \frac{2 (r-1)! \log 4}{\frac{1-\delta}{2(r-1)} \left(\frac{1-\delta}{2}\right)^{r-1}}.
\] 
(Let us note that we make no effort to optimize $c$ here.) In every step of the process, $V$ is partitioned into the following four sets:
\begin{itemize}
\item $P$ -- ordered set of vertices representing a loose path $(v_1, v_2, \ldots, v_{\ell})$ of length~$\ell$,
\item $U$ -- set of unexplored vertices that have not been reached by the search yet,
\item $W$ and $\tilde{W}$ -- sets of explored vertices that no longer play a role in the process.
\end{itemize}
We start the process by assigning to $P$ an arbitrary vertex $v$ from $V$ and setting $U=V \setminus \{v\}$, $W=\tilde{W}=\emptyset$. (Note that it is vacuously true that $L = (\{v\}, \emptyset)$, a hypergraph consisting of one vertex and no edge, is a loose path.) Suppose that at some point of the process we found a path $P = (v_1, v_2, \ldots, v_{\ell})$ of length $\ell$. If there exists a set $f=\{u_1, u_2, \dots,u_{r-1}\}\subseteq U$ of size $r-1$ such that $f \cup \{ v_{\ell} \} \in E$, then we extend the path by adding the edge $\{ v_{\ell}, u_1, u_2, \ldots, u_{r-1} \}$ to the path; vertices of $f$ are removed from $U$. (After this operation, $P=(v_1, v_2, \ldots, v_{\ell}, u_1, u_2, \ldots, u_{r-1})$ has length $\ell + (r-1)$.) Otherwise, that is, if no such set $f$ can be found in $U$, then $v_{\ell}$ is moved from $P$ to $W$ and, depending on the value of $\ell$, we have two cases. If $\ell \neq 1$ (which implies that, in fact, $\ell \ge r$), then vertices $v_{\ell-1},v_{\ell-2}, \dots, v_{\ell-(r-2)}$ are moved from $P$ to $\tilde{W}$. (After this operation, $P=(v_1, v_2, \ldots, v_{\ell-(r-1)})$ has length $\ell - (r-1)$.) If $\ell = 1$, then we simply take an arbitrary vertex from $U$ and restart the process by taking $P=(v)$ of length 1.

Now, let us assume that $H$ has no loose path of length at least $\delta n$. Therefore, in every step of the algorithm $\ell < \delta n$. Moreover, the process ends with $U=\emptyset$ and $|W|+|\tilde{W}| = |V|-|P| > (1-\delta)n$.
Hence, since $(r-2)|W|\ge |\tilde{W}|$, we have $|W| > \frac{1-\delta}{r-1} n$.
Observe that in each step of the process either the size of $W$ increases by one or the size of $U$ decreases by at most $r-1$ (actually, the size of $U$ decreases either by 1 or $r-1$). Thus, at some stage we must have $|W| = \frac{1-\delta}{2(r-1)} n$ and 
\[
|U| = |V|-|P|-|W| - |\tilde{W}|  > 
n - \delta n - |W| - (r-2)|W| 
= \frac{1-\delta}{2} n.
\]
Furthermore, note that there is no edge $e$ between $U$ and $W$ such that $|U\cap e|=r-1$ and $|W\cap e|=1$.

For fixed sets $U$ and $W$ satisfying  $|U| = \frac{1-\delta}{2} n$ and $|W| = \frac{1-\delta}{2(r-1)} n$, let $e(U,W)$ be the number of edges $e$ between $U$ and $W$ in $H$ for which $|U\cap e|=r-1$ and $|W\cap e|=1$.
Then,
\begin{align*}
\Prob(e(U,W) = 0)  &= (1-p)^{|W| \binom{|U|}{r-1}}\\
&\le \exp\left(-(1+o(1))\frac{c}{n^{r-1}} \cdot \frac{1-\delta}{2(r-1)} n \cdot \left(\frac{1-\delta}{2} n \right)^{r-1} /\, (r-1)! \right)\\
&= \exp( -(1+o(1)) 2n\log 4 ),
\end{align*}
and consequently by taking the union over all disjoint subsets $U,W\subseteq V$ satisfying $|U| = \frac{1-\delta}{2} n$ and $|W| = \frac{1-\delta}{2(r-1)} n $, we obtain
\begin{multline*}
\Prob\left(H \text{ has no loose path of length at least $\delta n$}\right)
\le \Prob\left( \bigcup_{U,W} e(U,W) = 0 \right) \\
\le 2^n 2^n \exp( -(1+o(1)) 2n\log 4 )
=\exp( -(1+o(1)) n\log 4 ) \le \exp(-n),
\end{multline*}
as required.
\end{proof}

In order to prove Theorem~\ref{thm:main3} we will also need the following result about Hamiltonian paths in $H_r(n,p)$, which follows from a series of papers~\cite{F,DF,DFLS} devoted to loose Hamiltonicity of random hypergraphs. Observe that if an $r$-uniform hypergraph of order~$n$ contains a loose Hamiltonian path (that is, a loose path that covers each vertex of $H$), then necessarily $r-1$ must divide $n-1$.

\begin{lemma}[\cite{F,DF,DFLS}]\label{lem:hamilton}
For fixed integer $r\ge 3$, let $p=p(n)$ be such that $p n^{r-1} / \log n$ tends to infinity together with~$n$ and $H\in H_r(n,p)$. Then, a.a.s.\ $H$ has a loose Hamiltonian path provided that $r-1$ divides $n-1$.
\end{lemma}

\section{Random graphs}\label{sec:graphs}\label{sec:graphs}

Let us start with the following two comments. First of all, note the following: whenever $G_2$ is a subgraph of $G_1$ on the same vertex set,  $\AC(G_1) \le \AC(G_2)$, since the agents in $G_1$ have more edges to use. Hence, Theorem~\ref{thm:main1} implies that for any $m \ge M$ we also have that a.a.s.\ $\AC(G) = O \left( n \right)$, $G \in \G(n,m)$.

Second of all, it is known that the two models ($G(n,p)$ and $\G(n,m)$) are in many cases asymptotically equivalent, provided ${n \choose 2}p$ is close to $m$. For example, Proposition~1.12 in~\cite{JLR} gives us the way to translate results from $\G(n,m)$ to $G(n,p)$. 

\begin{lemma}\label{lem:gnm_to_gnp}
Let $P$ be an arbitrary property, $p=p(n)$, and $c \in [0,1]$. If for every sequence $m=m(n)$ such that 
$$
m={n \choose 2}p + O \left( \sqrt{ {n \choose 2} p (1-p)} \right)
$$ 
it holds that $\Prob(\G(n,m) \in P) \to c$ as $n\to \infty$, then also $\Prob(G(n,p) \in P) \to c$ as $n \to \infty$.
\end{lemma}

Using this lemma, Theorem~\ref{thm:main1} implies immediately Theorem~\ref{thm:main2}. Indeed, suppose that $p=p(n)=(\log n + \omega)/n$, where $\omega=\omega(n)$ tends to infinity together with $n$. One needs to investigate $\G(n,m)$ for $m = \frac {n}{2} (\log n + \omega+o(1))$. It is known that in this range of $m$, a.a.s.\ $\G(n,m)$ is connected (that is, a.a.s.\ $M < m$). Theorem~\ref{thm:main1} together with the first observation imply that a.a.s.\ $\AC(\G(n,m)) = O \left( n \right)$, and Theorem~\ref{thm:main2} follows from Lemma~\ref{lem:gnm_to_gnp}.

\bigskip

In order to prove Theorem~\ref{thm:main1}, we will show that at the time when $\G(n,m)$ becomes connected (that is, at time $T$), a.a.s.\ it contains a certain spanning tree $T$ with $\AC(T)=O(n)$. For that, we need to introduce the following useful family of trees. A tree $T$ is \emph{good} if it consists of a path $P=(v_1,v_2,\ldots,v_k)$ (called the \emph{spine}), some vertices (called \emph{heavy}) that form a set $\{ u_i : i \in I \subseteq [k]\}$ that are connected to the spine by a perfect matching (that is, for every $i \in I$, $u_i$ is adjacent to $v_i$). All other vertices (called \emph{light}) are adjacent either to $v_i$ for some $i \in [k]$ or to $u_i$ for some $i \in I$ (see Figure~\ref{fig:tree} for an example).

\begin{figure}
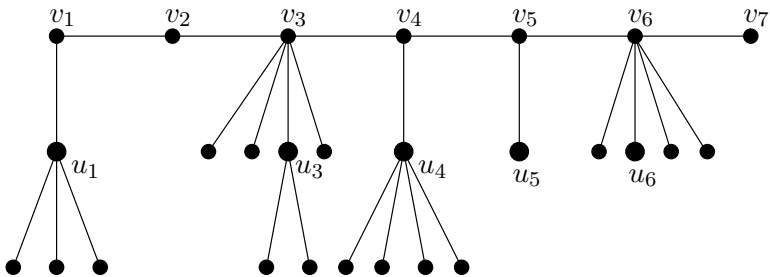
\caption{A good tree.}
\label{fig:tree}
\end{figure}

\smallskip

We will use the following result from~\cite{bst} (Claim 2.1 in that paper).
\begin{claim}[\cite{bst}]\label{claim:bst_route}
Let $G = (V,E)$ be a tree. Let $S, T \subseteq V$ be two subsets of the vertices of equal
size $k = |S| = |T|$, and let $\ell = \max_{v\in S,u\in T} {\dist(v, u)}$ be the maximal distance between a
vertex in $S$ and a vertex in $T$. Then, there is a strategy of $\ell + 2(k - 1)$ matchings that
routes all agents from $S$ to $T$.
\end{claim}

Now, we are ready to show that this family of trees is called good for a reason. 

\begin{lemma}\label{lem:structure}
Let $T$ be a good tree on $n$ vertices. Then, $\AC(G)=O(n)$ and, moreover, every agent visits every vertex of the spine.
\end{lemma}
\begin{proof}
Let us call agents occupying the spine $P=(v_1,v_2,\ldots,v_k)$ \emph{active}. First, we will show that there exists a strategy ensuring that within $6k$ rounds every active agent gets acquainted with everyone else and, moreover, that she visits every vertex of the spine.

For $1 \le i \le k-1$, let $e_i = v_iv_{i+1}$. The strategy is divided into $2k$ phases, each consisting of 3 rounds. On odd-numbered phases, swap active agents on all odd-indexed edges; then, swap every active agent adjacent to a heavy vertex (that is, agent occupying vertex $v_i$ for some $i \in I$) with non-active agent (occupying $u_i$), and then repeat it again so that all active agents are back onto the spine. On even-numbered phases, swap agents on all even-indexed edges, followed by two rounds of swapping with non-active agents, as it was done before. This has the following effect.  Agents that begin on odd-indexed vertices move ``forward'' in the vertex ordering, pause for one round at $v_k$, move ``backward'', pause again at $v_1$, and repeat; agents that begin on even-indexed vertices move backward, pause at $v_1$, move forward, pause at $v_k$, and repeat.  Since $T$ is good, every non-active agent was initially adjacent to some vertex of the spine or some heavy vertex. Hence, after $2k$ phases, each active agent has traversed the entire spine taking detours to heavy vertices, if needed; in doing so, she has necessarily passed by every other agent (regardless whether active or non-active). 

It remains to show that agents can be rotated so that each agent is active at some point. Note that the diameter of $T$ is $O(k)$.  Partition the $n$ agents into $\ceil{n/k}$ teams, each of size at most $k$.  We iteratively route a team onto the spine, call them active, apply the strategy described earlier for the active team, and repeat until all teams have traversed the spine.  Thus, by Claim~\ref{claim:bst_route}, each iteration can be completed in $O(k)$ rounds, so the total number of rounds needed is $\ceil{n/k} \cdot O(k) = O(n)$.
\end{proof}

As we already mentioned, our goal is to show that at the time when $\G(n,m)$ becomes connected, a.a.s.\ it contains a good spanning tree $T$. However, it is easier to work with $G(n,p)$ model instead of $G(n,m)$. Lemma~\ref{lem:gnm_to_gnp} provides us with a tool to translate results from $\G(n,m)$ to $G(n,p)$. The following lemma works the other way round, see, for example,~(1.6) in~\cite{JLR}.

\begin{lemma}\label{lem:gnp_to_gnm}
Let $P$ be an arbitrary property, let $m=m(n)$ be any function such that $m \le n \log n$, and take $p=p(n) = m/{n \choose 2}$. Then,
$$
\Prob(\G(n,m) \in P) \le 3 \sqrt{n \log n} \cdot \Prob(G(n,p) \in P).
$$
\end{lemma}

We will also need the following lemma.

\begin{lemma}\label{lem:technical}
Consider $G = (V,E) \in G(n,\log \log \log n / n)$. Then, for every set $A$ of size $0.99n$ the following holds. For a set $B$ of size $O(n^{0.03})$ taken uniformly at random from $V \setminus A$, $B$ induces a graph with no edge with probability $1-o((n \log n)^{-1/2})$.
\end{lemma}
\begin{proof}
Fix a vertex $v \in V$. The expected degree of $v$ is $(1+o(1)) \log \log \log n$. It follows from Bernstein's bound~(\ref{Bernstein}), applied with $x = c \log n / \log \log \log n$ for $c$ large enough, that with probability $o(n^{-2})$, $v$ has degree larger than, say, $2 c \log n$. Thus, the union bound over all vertices of $G$ implies that with probability $1-o((n \log n)^{-1/2})$ all vertices have degrees at most $2 c \log n$. Since we aim for such a probability, we may assume that $G$ is a deterministic graph that has this property.

Now, fix any set $A$ of size $0.99n$. Regardless of our choice of $A$, clearly, every vertex of $V \setminus A$ has at most $2 c \log n$ neighbours in $V \setminus A$. Take a random set $B$ from $V \setminus A$ of size $O(n^{0.03})$ and fix a vertex $v \in B$. The probability that no neighbour of $v$ is in $B$ is at least
$$
\frac{ { {0.01 n -1 - 2 c \log n} \choose {|B|-1} } } { { {0.01 n-1} \choose {|B|-1} } } = 1 - O \left( \frac {n^{0.03} \log n} {n} \right) = 1 - O \left( n^{-0.97} \log n\right).
$$
Hence, the probability that some neighbour of $v$ is in $B$ is $O \left( n^{-0.97} \log n\right)$. Consequently, the union bound over all vertices of $B$ implies that with probability $O \left( n^{-0.94} \log n\right)$ there is at least one edge in the graph induced by $B$. The proof of the lemma is finished.
\end{proof}

Finally, we are ready to prove the main result of this paper.

\begin{proof}[Proof of Theorem~\ref{thm:main1}]
Let $M_- = \frac n2 (\log n - \log \log n)$ and let $M_+ = \frac n2 (\log n + \log \log n)$; recall that a.a.s.\ $M_- < M < M_+$. First, we will show that a.a.s.\ $\G(n,M_-)$ consists of a good tree $T$ and a small set $S$ of isolated vertices. Next, we will show that between time $M_-$ and $M_+$ a.a.s.\ no edge is added between $S$ and light vertices of $T$. This will finish the proof, since $\G(n,M)$ is connected and so at that point of the process, a.a.s.\ vertices of $S$ must be adjacent to the spine or heavy vertices of $T$, which implies that a.a.s.\ there is a good spanning tree. The result will follow then from Lemma~\ref{lem:structure}.

As promised, let $p_- = M_-/{n\choose 2} = (\log n - \log \log n+o(1))/n$ and consider $G(n,p_-)$ instead of $\G(n,M_-)$. In order to avoid technical problems with events not being independent, we use a classic technique known as \emph{two-round exposure}. The observation is that a random graph $G \in G(n,p)$ can be viewed as a union of two independently generated random graphs $G_1 \in G(n,p_1)$ and $G_2 \in G(n,p_2)$, provided that $p=p_1 + p_2 - p_1 p_2$ (see, for example,~\cite{bol, JLR} for more information). 

Let $p_1 := \log \log \log n / n$ and
$$
p_2 := \frac {p_--p_1}{1-p_1} = \frac{\log n - \log \log n - \log \log \log n + o(1)}{n}.
$$
Fix $G_1 \in G(n,p_1)$ and $G_2 \in G(n,p_2)$, with $V(G_1) = V(G_2)=V$, and view $G \in G(n,p_-)$ as the union of $G_1$ and $G_2$.  It follows from Lemma~\ref{lem:alg} that with probability $1-o((n \log n)^{-1/2})$, $G_1$ has a long path  $P=(v_1, v_2, \ldots, v_k)$ of length $k=0.99n$ (and thus $G$ has it too). This path will eventually become the spine of the spanning tree. 

Now, we expose edges of $G_2$ in a very specific order dictated by the following algorithm. Initially, $A$ consists of vertices of the path (that is, $A = \{v_1, v_2, \ldots, v_k\}$), $B = \emptyset$, and $C = V \setminus A$. At each step of the process, we take a vertex $v$ from $C$ and expose edges from $v$ to $A$. If an edge from $v$ to some $v_i$ is found, we change the label of $v$ to $u_i$, call it heavy, remove $v_i$ from $A$, and move $u_i$ from $C$ to $B$. Otherwise (that is, if there is no edge from $v$ to $A$), we expose edges from $v$ to $B$. If an edge from $v$ to some heavy vertex $u_i$ is found, we call $v$ light, and remove it from $C$. Clearly, at the end of this process we are left with a small (with the desired probability, as we will see soon) set $C$ and a good tree $T_1$. Let $X = |C|$ be the random variable counting vertices not attached to the tree yet. Since at each step of the process $|A|+|B|$ is precisely $0.99n$, we have
\begin{eqnarray*}
\E X &=& 0.01 n \big( 1 - p_2 \big)^{0.99n} \\
&=& (0.01+o(1)) \exp \big ( \log n - 0.99 ( \log n - \log \log n - \log \log \log n) \big)\\
&=& (0.01+o(1)) n^{0.01} \big( (\log n)(\log \log n) \big)^{0.99}.
\end{eqnarray*}
Note that $X$ is, in fact, a binomial random variable $\textrm{Bin}(0.01n,( 1 - p_2 )^{0.99n})$. Hence, it follows from Chernoff's bound~(\ref{chern}) that $X = (1+o(1)) \E X$ with probability $1-o((n \log n)^{-1/2})$. Now, let $Y$ be the random variable counting how many vertices \emph{not} on the path are \emph{not} heavy.  Arguing as before, since at each step of the process $|A| \ge 0.98n$, we have
$$
\E Y \le 0.01 n \big( 1 - p_2 \big)^{0.98n} = (0.01+o(1)) n^{0.02} \big( (\log n)(\log \log n) \big)^{0.98}.
$$
Clearly, $Y \ge X$ and so $\E Y$ is large enough for Chernoff's bound~(\ref{chern}) to be applied again to show that $Y = (1+o(1)) \E Y$ with probability $1-o((n \log n)^{-1/2})$.

Our next goal is to attach almost all vertices of $C$ to $T_1$ in order to form another, slightly larger, good tree $T_2$. Consider a vertex $v \in C$ and a heavy vertex $u_i$ of $T_1$. Obvious, yet important, property is that when edges emanating from $v$ were exposed in the previous phase, exactly one vertex from ${v_i, u_i}$ was considered (recall that when $u_i$ is discovered as a heavy vertex, $v_i$ is removed from $A$). Hence we may expose these edges and try to attach $v$ to the tree. From the previous argument, since we aim for a statement that holds with probability $1-o((n \log n)^{-1/2})$, we may assume that the number of heavy vertices is at least $0.01n - n^{0.03}$. For the random variable $Z$ counting vertices still not attached to the path, we get
\begin{eqnarray*}
\E Z &=& (1+o(1)) \E X \big( 1 - p_2 \big)^{0.01n - O(n^{0.03})} \\
&=& (0.01+o(1)) n^{0.01} \big( (\log n)(\log \log n) \big)^{0.99} \\
&& \quad \times \exp \big ( - 0.01 ( \log n - \log \log n - \log \log \log n) \big)\\
&=& (0.01+o(1)) \big( (\log n)(\log \log n) \big),
\end{eqnarray*}
and so $Z = (1+o(1)) \E Z$ with probability $1-o((n \log n)^{-1/2})$ by Chernoff's bound~(\ref{chern}).

Let us stop for a second and summarize the current situation. We showed that with the desired probability, we have the following structure. There is a good tree $T_2$ consisting of all but at most $(\log n)(\log \log n)$ vertices that form a set $S$. $T_2$ consists of the spine of length $0.99n$,  $0.01n - O(n^{0.03})$ heavy vertices and $O(n^{0.03})$ light vertices. Edges between $S$ and the spine and between $S$ and heavy vertices are already exposed and no edge was found. On the other hand, edges within $S$ and between $S$ and light vertices are not exposed yet. However, it is straightforward to see that with the desired probability there is no edge there neither, and so vertices of $S$ are, in fact, isolated in $G_2$. Indeed, the probability that there is no edge in the graph induced by $S$ and light vertices is equal to
$$
\big( 1 - p_2 \big)^{O(n^{0.06})} = \exp \big( - O(n^{0.06} \log n / n) \big) = 1 - o( (n \log n)^{-1/2} ).
$$

Finally, we need to argue that vertices of $S$ are also isolated in $G_1$. The important observation is that the algorithm we performed that exposed edges in $G_2$ used only the fact that vertices of $\{v_1, v_2, \ldots, v_{0.99n}\}$ form a long path; no other information about $G_1$ was ever used. Hence, set $S$ together with light vertices is, from the perspective of the graph $G_1$, simply a random set of size $O(n^{0.03})$ taken from the set of vertices not on the path. Lemma~\ref{lem:technical} implies that with probability $1-o((n \log n)^{-1/2})$ there is no edge in the graph induced by this set. With the desired property, $G(n,p)$ consists of a good tree $T$ and a small set $S$ of isolated vertices and so, by Lemma~\ref{lem:gnp_to_gnm}, a.a.s.\ it is also true in $\G(n,M_-)$.

It remains to show that between time $M_-$ and $M_+$ a.a.s.\ no edge is added between $S$ and light vertices of $T$. Direct computations for the random graph process show that this event holds with probability
\begin{eqnarray*}
{{ {n \choose 2} - M_- - O(n^{0.06})} \choose {M_+ - M_-} } / {{ {n \choose 2} - M_-} \choose {M_+ - M_-} } &=& \prod_{i=0}^{M_+-M_--1} \frac { {n \choose 2} - M_- - O(n^{0.06}) - i}{{n \choose 2} - M_- -i}\\
&=& \left( 1 - \frac {O(n^{0.06})}{n^2} \right)^{M_+ - M_-} \\
&=& 1 - \frac {O(n^{0.06} \log \log n)}{n}  \to 1
\end{eqnarray*}
as $n \to \infty$, since $M_+ - M_- = O(n \log \log n)$. The proof is finished.
\end{proof}

\section{Random hypergraphs}\label{sec:hypergraphs}

First, let us mention that the trivial bound for graphs~(\ref{eq:trivial_lower}) can easily be generalized to $r$-uniform hypergraphs and any $2 \le k\le r$:
\[
\AC^k_r(H) \ge \frac {{|V| \choose k}}{|E| {r \choose k}} - 1.
\]
For $H = (V,E) \in H_r(n,p)$ with $p > (1+\eps) \frac {(r-1)! \log n}{n^{r-1}}$ for some $\eps > 0$, we get immediately that a.a.s.
$$
\AC^k_r(H) = \Omega \left( \frac {1}{p n^{r-k}} \right),
$$
since the expected number of edges in $H$ is ${n \choose r} p = \Omega(n \log n)$ and so the concentration follows from Chernoff's bound~(\ref{chern}). Hence, the lower bound in Theorem~\ref{thm:main3} holds.

Now we prove the upper bound. In order to do it, we will split Theorem~\ref{thm:main3} into two parts and then, independently, prove each of them.

\begin{theorem_a}
Let $r \in \N \setminus \{1\}$, let $k \in \N$ be such that $2 \le k \le r$, and let $\eps > 0$ be any real number. Suppose that $p=p(n) = (1+\eps) \frac {(r-1)! \log n}{n^{r-1}}$. For $H \in H_r(n,p)$, a.a.s.
$$
\AC^k_r(H) = O(n^{k-1}) = O \left( \frac {\log n}{p n^{r-k}} \right).
$$
\end{theorem_a}

\begin{theorem_b}
Let $r \in \N \setminus \{1\}$, let $k \in \N$ be such that $2 \le k \le r$, and let $\omega = \omega(n)$ be any function tending to infinity together with $n$. Suppose that $p=p(n) = \omega \frac {\log n}{n^{r-1}}$. For $H \in H_r(n,p)$, a.a.s.
$$
\AC^k_r(H) = O \left( \max \left\{ \frac {\log n}{p n^{r-k}}, 1 \right\} \right).
$$
\end{theorem_b}

Note that when, say, $p \ge 2 k (r-k)! \frac {\log n}{n^{r-k}}$, the expected number of $k$-tuples that do \emph{not} get acquainted initially (that is, those that are not contained in any hyperedge) is equal to
$$
{n \choose k} (1-p)^{n-k \choose r-k} \le \exp \left( k \log n - (1+o(1)) 2k (r-k)! \frac {\log n}{n^{r-k}} \cdot \frac {n^{r-k}}{(r-k)!} \right) = o(1),
$$
and so, by Markov's inequality, a.a.s.\ every $k$-tuple gets acquainted immediately and $\AC^k_r(H)=0$. This is also the reason for taking the maximum of the two values in the statement of the result. 

\bigskip

For a given hypergraph $H=(V,E)$, sometimes it will be convenient to think about its \emph{underlying graph} $\tilde{H}=(V,\tilde{E})$ with $\{u,v\} \in \tilde{E}$ if and only if there exists $e \in E$ such that $\{u, v\} \subseteq e$. Clearly, two agents are matched in $\tilde{H}$ if and only if they are matched in~$H$. Therefore, we may always assume that agents walk on the underlying graph of~$H$.

\smallskip

Let $H=(V,E)$ be an $r$-uniform hypergraph. A \emph{1-factor} of $H$ is a set $F\subseteq E$ such that every vertex of $V$ belongs to exactly one edge in $F$. Moreover, a \emph{1-factorization} of $H$ is a partition of $E$ into 1-factors, that is, $E=F_1\cup F_2 \cup \dots\cup F_{\ell}$, where each $F_i$ is a 1-factor and $F_i \cap F_j = \emptyset$ for all $1\le i < j \le \ell$. Denote by $K_n^r$ the \emph{complete} $r$-uniform hypergraph of order~$n$ (that is, each $r$-tuple is present). We will use the following well-known result of Baranyai~\cite{Bar}.

\begin{lemma}[\cite{Bar}]\label{lem:factors}
If $r$ divides $n$, then $K_n^r$ is 1-factorizable.
\end{lemma}
\noindent
(Clearly the number of  1-factors is equal to $\binom{n}{r}\frac{r}{n} = {n-1 \choose r-1}$.)
The result on factorization is helpful to deal with loose paths.

\begin{lemma}\label{lem:ac_path}
Let $r \in \N \setminus \{1\}$ and let $k \in \N$ be such that $2 \le k \le r$. For a loose $r$-uniform path $P_n$ on $n$ vertices, 
$$
\AC^k_r (P_n) = O(n^{k-1}).
$$
Moreover, every agent visits every vertex of the path.
\end{lemma}
\begin{proof}
First, we claim that the number of steps, $f(n)$, needed for agents to be moved to any given permutation of $[n]$, the vertex set of $P_n$, satisfies $f(n)=O(n)$. For that we will use the fact that the underlying graph $\tilde{P_n}$ contains a classic path (graph) $\bar{P_n}$ of length $n$ as a subgraph. Now, $\bar{P_n}$ can be split into two sub-paths of lengths $\lceil n/2 \rceil$ and $\lfloor n/2 \rfloor$, respectively, and then agents are directed to the corresponding sub-paths that takes $O(n)$ steps by Claim~\ref{claim:bst_route}. We proceed recursively, performing operations on both paths (at the same time) to get that
$$
f(n) = O(n) + f(\lceil n/2 \rceil) = O(n).
$$
The claim holds.

Let $N$ be the smallest integer such that $N \ge n$ and $N$ is divisible by $k-1$. Clearly, $N-n \le k-2$. Lemma~\ref{lem:factors} implies that $K_N^{k-1}$ on the vertex set $V(K_N^{k-1}) = [N] \supseteq [n] = V(P_n)$ is 1-factorizable. Let $E_i$ for $1\le i \le \binom{N-1}{k-2}$ be the corresponding 1-factors. Obviously, $|E_i| = \frac{N}{k-1}$ for each $i$. Let $F_i  = \{ e \cap [n] : e\in E_i \}$. This way we removed all vertices that are not in $P_n$ and so $F_i \subseteq V(P_n)$ for each $1\le i \le \binom{N-1}{k-2}$ and every $k$-tuple from $V(P_n)$ is in some $F_i$. Call every member of $F_i$ a \emph{team}. Moreover, let us mention that, since at most $k-2$ vertices are removed, no team reduces to the empty set and so $|F_i| = \frac{N}{k-1}$ for each $i$.

Let us focus for a while on $F_1$. By the above claim we can group all teams of $F_1$ on $P_n$ in $O(n)$ rounds (using an arbitrary permutation of teams, one team next to the other; see also Remark~\ref{rem:comment}). We treat teams of agents as \emph{super-vertices} that form an (abstract) \emph{super-path} of length $\ell = |F_1| = N/(k-1)$. As in the proof of Lemma~\ref{lem:structure}, we use the strategy ensuring that within $O(n)$ rounds every team of agents gets acquainted with every other agent and, moreover, that every agent visits every vertex of $P$. Let $u_i$ be the $i$-th super-vertex corresponding to the $i$-th team. For $1 \le i \le \ell-1$, let $e_i = u_iu_{i+1}$. The strategy is divided into $2\ell$ phases, each consisting of $O(1)$ rounds. On odd-numbered phases, swap teams agents on all odd-indexed edges. On even-numbered phases, swap teams agents on all even-indexed edges. This guarantees that every pair of teams eventually meets, as discussed in the proof of Lemma~\ref{lem:structure}. The only difference is that this time swapping teams usually cannot be done in one round but clearly can be done in $O(1)$ rounds. (For example, the first agent from the second team swaps with each agent from the first team to move to the other side. This takes at most $k-1$ swaps. Then the next agent from the second team repeats the same, and after at most $(k-1)^2$ rounds we are done.) Moreover, this procedure can be done in such a way that each team meets each agent from the passing team on some hyperedge of $P_n$. This might require additional $O(1)$ rounds but can easily be done.

After $O(n)$ steps, each team from $F_1$ consisting of $k-1$ agents, is $k$-acquainted with each agent. Repeating this process for each $F_i$ yields the desired strategy for the total of $\binom{N-1}{k-2} \cdot O(n) = O(n^{k-1})$ rounds. 
\end{proof}

\begin{rem}\label{rem:comment}
As we already mentioned, in the previous proof, the order in which teams are distributed for a given 1-factor was not important. One can fix an arbitrary order or, say, a random one that will turn out to be a convenient way of dealing with some small technical issue in the proof of Theorem~\ref{thm:main3}b.
\end{rem}

\smallskip

It remains to prove the two theorems.

\begin{proof}[Proof of Theorem~\ref{thm:main3}a]
As in the proof of Theorem~\ref{thm:main1}, we use two-round exposure technique: $H \in H_r(n,p)$ can be viewed as the union of two independently generated random hypergraphs $H_1 \in H_r(n,p_1)$ and $H_2 \in H_r(n,p_2)$, with $p_1 := \frac {\eps}{2} \cdot \frac {(r-1)!\log n}{n^{r-1}}$ and 
$$
p_2 := \frac {p-p_1}{1-p_1} \ge p-p_1 = \left(1 + \frac {\eps}{2} \right) \frac{(r-1)!\log n}{n^{r-1}}.
$$
It follows from Lemma~\ref{lem:alg} that a.a.s.\ $H_1$ has a long loose path  $P$ of length $\delta n$ (and thus $H$ has it, too), where $\delta \in (0,1)$ can be made arbitrarily close to 1, as needed. 

Now, we will use $H_2$ to show that a.a.s.\ every vertex not on the path $P$ belongs to at least one edge of $H_2$ (and thus $H$, too) that intersects with $P$. Indeed, for a given vertex $v$ not on the path $P$, the number of $r$-tuples containing $v$ and at least one vertex from $P$ is equal to
\begin{eqnarray*}
s &=& {n - 1 \choose r-1} - {(1-\delta)n - 1 \choose r-1} = (1+o(1)) \left( \frac {n^{r-1}} {(r-1)!} - \frac {((1-\delta)n)^{r-1}} {(r-1)!} \right) \\
&=& (1+o(1)) \frac {n^{r-1}} {(r-1)!} \left( 1 - (1-\delta)^{r-1} \right),
\end{eqnarray*}
and so the probability that none of them occurs as an edge in $H_2$ is
$$
(1-p_2)^s \le \exp \left( - (1+o(1)) \left(1 + \frac {\eps}{2} \right) \frac{(r-1)!\log n}{n^{r-1}} \cdot \frac {n^{r-1}} {(r-1)!} \left( 1 - (1-\delta)^{r-1} \right) \right) = o(n^{-1}),
$$
after taking $\delta$ close enough to 1. The claim holds by the union bound.

The rest of the proof is straightforward. We partition all agents into $(k+1)$ groups, $A_1, A_2, \ldots, A_{k+1}$, each consisting of at least $\lfloor n/(k+1) \rfloor$ agents. Since $\delta$ can be arbitrarily close to 1, we may assume that $\lceil n/(k+1) \rceil> (1-\delta) n$. For every $i \in [k+1]$, we route \emph{all} agents from all groups but $A_i$ and some agents from $A_i$ to $P$, and after $O(n^{k-1})$ rounds they get acquainted by Lemma~\ref{lem:ac_path}. Since each $k$-tuple of agents intersects at most $k$ groups, it must get acquainted eventually, and the proof is complete. 
\end{proof}

\begin{proof}[Proof of Theorem~\ref{thm:main3}b]
First we assume that $\omega = \omega(n)  \le n^{k-1}$ and $r-1$ divides $n-1$. As usual, we consider two independently generated random hypergraphs $H_1 \in H_r(n,p_1)$ and $H_2 \in H_r(n,p_2)$, with $p_1 := \frac {\omega}{2} \cdot \frac {\log n}{n^{r-1}}$ and $p_2 > p-p_1 = p_1$. It follows from Lemma~\ref{lem:hamilton} that a.a.s.\ $H_1$ has a Hamiltonian path $P=(v_1, v_2, \ldots, v_n)$. 

Let $C$ be a large constant that will be determined soon. We cut $P$ into $(\omega/C)^{1/(k-1)}$ loose sub-paths, each of length $\ell := (C/\omega)^{1/(k-1)} n$; there are $r-2$ vertices between each pair of consecutive sub-paths that are not assigned to any sub-path for a total of at most $r \ell = r \cdot (\omega/C)^{1/(k-1)} \le r n C^{-1/(k-1)}$ vertices,  since $\omega \le n^{k-1}$. This can be made smaller than, say, $\lceil n/(k+1) \rceil$, provided that $C$ is large enough. We will call agents occupying these vertices to be \emph{passive}; agents occupying sub-paths will be called \emph{active}. Active agents will be partitioned into units, with respect to the sub-path they occupy. Every unit performs (independently and simultaneously) the strategy from Lemma~\ref{lem:ac_path} on their own sub-path. After $s := O(C n^{k-1}/\omega)$ rounds, each $k$-tuple of agents from the same unit will be $k$-acquainted. 

Now, we will show that a.a.s.\ every $k$-tuple of active agents, not all from the same unit, gets acquainted. Let us focus on one such $k$-tuple, say $f$. If one starts from a random order in the strategy from Lemma~\ref{lem:ac_path} (see Remark~\ref{rem:comment}), then it is clear that with probability $1-o(n^{k})$ agents from $f$ visit at least $s / 2 = Cn^{k-1}/(2\omega)$ different $k$-tuples of vertices. Hence, by considering vertices that are not on paths agents from $f$ walk on, we get that with probability $1-o(n^{k})$, the number of distinct $r$-tuples visited by agents from $f$ is at least 
$$
t := \frac {s}{2} {n - k \ell \choose r-k} \ge \frac {Cn^{k-1}}{2\omega} \cdot \frac { (0.9 n)^{r-k} }{ (r-k)!} 
= \frac {C (0.9)^{r-k} }{2(r-k)!} \cdot \frac { n^{r-1} }{ \omega}.
$$
Considering only those edges in $H_2$, the probability that agents from $f$ never got acquainted is at most
$$
(1-p_2)^{t} \le \exp \left( - \frac {\omega \log n}{n^{r-1}} \cdot \frac {C (0.9)^{r-k} }{2(r-k)!} \cdot \frac { n^{r-1} }{ \omega} \right) = \exp \left( - \frac {C (0.9)^{r-k} }{2(r-k)!} \log n \right) = o(n^{-k}),
$$
for $C$ sufficiently large. Since there are at most ${n \choose k} = O(n^k)$ $k$-tuples of agents, a.a.s.\ all active agents got acquainted.

It remains to deal with passive agents which can be done the same way as in the proof of Theorem~\ref{thm:main3}a. We partition all agents into $(k+1)$ groups, and shuffle them so that all get acquainted. Moreover, if $r-1$ does not divide $n-1$, we may take $N$ to be the largest integer such that $N \le n$ and $r-1$ divides $N-1$, and deal with a loose path on $N$ vertices. Clearly, a.a.s.\ each of the remaining $O(1)$ vertices belong to at least one hyperedge intersecting with $P$. We may treat these vertices as passive ones and proceed as before.

Finally, suppose that $\omega = \omega(n) > n^{k-1}$. Since $\AC^k_r( H_r(n,p) )$ is a decreasing function of $p$, 
$$
\AC^k_r( H_r(n,p) ) \le \AC^k_r( H_r(n, (\log n) / n^{r-k}) ) = O(1),
$$ 
by the previous case and the theorem is finished.
\end{proof}

\section{Concluding remarks}

In this paper, we showed that as soon as the random graph process $\G(n,m)$ becomes connected its acquaintance time is linear, which implies that for $G \in G(n,p)$ the following property holds a.a.s.: if $G$ is connected, then $\AC(G) = O(\log n / p)$. This confirms the conjecture from~\cite{bst} in the strongest possible sense. On the other hand, the trivial lower bound~(\ref{eq:trivial_lower}) is off by a factor of $\log n$, namely, for $G \in G(n,p)$ the following property holds a.a.s.: if $G$ is connected, then $\AC(G) = \Omega(1 / p)$. It is worth to mentioning that it was shown in~\cite{KMP} that for dense random graphs (that is, graphs with average degree at least $n^{1/2+\eps}$ for $\eps>0$) the upper bound gives the right order of magnitude. It would be interesting to strengthen these bounds and investigate the order of $\AC(G)$ for sparse random graphs.

Of course, an analogous question could be asked (and probably should) for hypergraphs. It would be desired to improve the bounds in Theorem~\ref{thm:main3}. Moreover, in Theorem~\ref{thm:main3}, it is assumed that $p$ is by a multiplicative factor of $(1+\eps)$ larger than the threshold for connectivity ($\eps>0$ is any constant, arbitrarily small). It remains to investigate the behaviour of $\AC^k_r(H)$ in the critical window, preferably by investigating the random hypergraph process as it was done for graphs in Theorem~\ref{thm:main1}.

It might also be of some interest to study the acquaintance time of other models of random graphs such as, for example, power-law graphs, preferential attachments graphs, or geometric graphs. The latter was recently considered  in~\cite{MP}. Moreover, let us mention that there are some number of interesting open problems for deterministic graphs, see, for example~\cite{bst}.


\begin{thebibliography}{99}

\bibitem{ACG94} N.~Alon, F.R.K.~Chung, and R.L.~Graham. Routing permutations on graphs via matchings, \emph{SIAM J.\ Discrete Math.}\ \textbf{7} (1994), 513--530.

\bibitem{AS} O.\ Angel and I.\ Shinkar, A tight upper bound on acquaintance time of graphs, preprint.

\bibitem{Bar}Zs.~ Baranyai, 
On the factorization of the complete uniform hypergraph. In: Infinite and finite sets, Colloq. Math. Soc. Janos Bolyai, Vol. 10, North-Holland, Amsterdam, 1975, pp. 91--108.
 
\bibitem{bst} I.~Benjamini, I.~Shinkar, and G.~Tsur, Acquaintance time of a graph, preprint.

\bibitem{bol} B.\ Bollob\'{a}s, \emph{Random Graphs}, Cambridge University Press, Cambridge, 2001.

\bibitem{Che09} N.~Chen, On the approximability of influence in social networks, \emph{SIAM J.\ Discrete Math.} \textbf{23(5)} (2009), 1400--1415.

\bibitem{DF} A.~Dudek and A.~Frieze, Loose Hamiltonian cycles in random uniform hypergraphs, \emph{Electron.\ J.\ Combin.}\ \textbf{18} (2011), \#P48.

\bibitem{DFLS} A.~Dudek, A.~Frieze, P.-S.~Loh, and S.~Speiss, Optimal divisibility conditions for loose Hamilton cycles in random hypergraphs, \emph{Electron.\ J.\ Combin.}\ \textbf{19} (2012), \#P44.

\bibitem{F} A.~Frieze, Loose Hamilton cycles in random 3-uniform hypergraphs, \emph{Electron.\ J.\ Combin.}\ \textbf{17} (2010), \#N283.

\bibitem{HHL88} S.T.~Hedetniemi, S.M.~Hedetniemi, and A.~Liestman. A survey of gossiping and broadcasting in communication networks, \emph{Networks} \textbf{18(4)} (1998), 319--349.

\bibitem{JLR} S.\ Janson, T.\ {\L}uczak, A.\ Ruci\'nski, \emph{Random Graphs}, Wiley, New York, 2000.

\bibitem{KKT03} D.~Kempe, J.~Kleinberg, and E.~Tardos, Maximizing the spread of influence through a social network, In KKD, 137--146, 2003.

\bibitem{KMP} W.\ Kinnersley, D.\ Mitsche, and P.\ Pra\l{}at, A note on the acquaintance time of random graphs, \emph{Electron.\ J.\ Combin.}\ \textbf{20(3)} (2013), \#P52.

\bibitem{KLS} M.~Krivelevich, C.~Lee, and B.~Sudakov, Long paths and cycles in random subgraphs of graphs with large minimum degree, to appear in \emph{Random Structures \& Algorithms}.

\bibitem{MP} T.\ M\"{u}ller, P.\ Pra\l{}at, The acquaintance time of (percolated) random geometric graphs, preprint.

\bibitem{Rei12} D.~Reichman, New bounds for contagious sets, \emph{Discrete Math.}\ \textbf{312} (2012), 1812--1814.

\end{thebibliography}
\end{document}